\documentclass[sn-basic]{sn-jnl}

\usepackage{graphicx}%
\usepackage{multirow}%
\usepackage{amsmath,amssymb,amsfonts}%
\usepackage{amsthm}%
\usepackage{mathrsfs}%
\usepackage{tikz-cd}
\usepackage[title]{appendix}%
\usepackage{xcolor}%
\usepackage{textcomp}%
\usepackage{manyfoot}%
\usepackage{booktabs}%
\usepackage{subcaption}

\theoremstyle{plain}
\newtheorem{theorem}{Theorem}[section]
\newtheorem{lemma}[theorem]{Lemma}
\newtheorem{proposition}[theorem]{Proposition}

\theoremstyle{definition}
\newtheorem{definition}[theorem]{Definition}

\newcommand{\wh}{\widehat}

\newcommand{\bitem}{\begin{itemize}}
\newcommand{\eitem}{\end{itemize}}
\newcommand{\mc}[1]{\mathcal{#1}}

\newcommand{\II}{\mathbb{I}}

\newcommand{\N}{\mathbb{N}}
\newcommand{\R}{\mathbb{R}}

\newcommand{\bpm}{\begin{pmatrix}}
\newcommand{\epm}{\end{pmatrix}}
\newcommand{\bvm}{\begin{vmatrix}}
\newcommand{\evm}{\end{vmatrix}}
\newcommand{\bsm}{\left(\begin{smallmatrix}}
\newcommand{\esm}{\end{smallmatrix}\right)}
\newcommand{\T}{\top}

\newcommand{\ol}[1]{\overline{#1}}
\newcommand{\la}{\langle}
\newcommand{\ra}{\rangle}

\newcommand{\mrm}[1]{\mathrm{#1}}

\newcommand{\veps}{\varepsilon}

\newcommand{\dd}{\textup{d}}

\newcommand{\eins}{\mathbf{1}}

\DeclareMathSymbol{\mydiv}{\mathbin}{symbols}{"04}

\DeclareMathOperator{\Diag}{Diag}

\DeclareMathOperator{\supp}{supp}

\DeclareMathOperator{\vvec}{vec}

\DeclareMathOperator{\mrank}{rank}
\DeclareMathOperator{\mkernel}{ker}
\DeclareMathOperator{\mimg}{img}
\DeclareMathOperator{\NE}{NE}

\DeclareMathOperator{\KL}{KL}

\newcommand{\p}{{p}}

\newcommand{\SM}{{\mc{S}}}
\newcommand{\Sn}{{\mc{S}_N}}
\newcommand{\WM}{{\mc{W}}}
\newcommand{\BS}{{\eins_\SM}}
\newcommand{\BW}{{\eins_\WM}}

\newcommand{\TW}{{T_0\WM}}

\newcommand{\ROS}{{R}}
\newcommand{\ROW}{{\mc{R}}}

\newcommand{\Graph}{{\mc{G}}}
\newcommand{\Nodes}{{\mc{V}}}
\newcommand{\Edges}{{\mc{E}}}

\newcommand{\imT}{{\mc{T}}}
\newcommand{\Margin}{M} %\mc{M}

\begin{document}

\title[]{A Geometric Embedding Approach to Multiple Games and Multiple Populations}

\author*[1]{\fnm{Bastian} \sur{Boll}}\email{bastian.boll@iwr.uni-heidelberg.de}

\author[1]{\fnm{Jonas} \sur{Cassel}}

\author[1]{\fnm{Peter} \sur{Albers}}

\author[2]{\fnm{Stefania} \sur{Petra}}

\author[1]{\fnm{Christoph} \sur{Schn\"orr}}

\affil[1]{\orgdiv{Institute for Mathematics}, \orgname{Heidelberg University}, \orgaddress{\street{Im Neuenheimer Feld 205}, \city{Heidelberg}, \postcode{69120}, \state{Baden-W\"urttemberg}, \country{Germany}}}

\affil[2]{\orgdiv{Institute for Mathematics}, \orgname{Augsburg University}, \orgaddress{\street{Universit\"{a}tsstra{\ss}e 2}, \city{Augsburg}, \postcode{86159}, \state{Bavaria}, \country{Germany}}} % Mathematical Image Analysis Group

\abstract{This paper studies a meta-simplex concept and geometric embedding framework for multi-population replicator dynamics. Central results are two embedding theorems which  constitute a formal reduction of multi-population replicator dynamics to single-population ones. In conjunction with a robust mathematical formalism, this provides a toolset for analyzing complex multi-population models. Our framework  provides a unifying perspective on different population dynamics in the literature which in particular enables to establish a formal link between multi-population and multi-game dynamics.}

\keywords{Replicator dynamics, Assignment Flows, Manifold Embedding, Information Geometry}

\maketitle

\tableofcontents

\section{Introduction}\label{sec:intro}

\subsection{Overview, Contribution}
Evolutionary game theory \citep{Hofbauer:1998, Sandholm2010} is an established framework for modeling problems in diverse areas ranging from mathematical biology \citep{Smith:1973,Smith:1982,Hammerstein:1994,Nowak:2006aa,Leimar:2023} to economics \citep{Gardner:2003,Samuelson:2016}.
It assumes a dynamic perspective on games played by a large and well-mixed population of agents. In this context, the earliest dynamical model of population state is the \emph{replicator equation} \citep{Taylor:1978, Schuster:1983}, which has since been generalized in several ways \citep{Bednar:2007,Cressman:2014} to accommodate more complex situations.

This paper provides an embedding approach for studying the following two classes of scenarios within a single framework.
\begin{itemize}
\item
\emph{Multi-population dynamics} model multiple interacting populations or species. The state space is a product manifold of multiple simplices, and payoff may depend on the state of all populations. The resulting dynamics are multiple coupled replicator dynamics on the product manifold of multiple simplices.
\item
\emph{Multi-game dynamics} model agents that simultaneously play multiple games, earning cumulative payoff. The state space is a single simplex with dimension growing exponentially in the number of games. Interaction between games occurs whenever the population state is outside of a specific submanifold.
\end{itemize}

Here, we study an embedding of multiple probability simplices into a combinatorially large simplex of joint distributions. As further detailed in Section \ref{sec:conclusion}, our approach is closely related to \emph{Segre embeddings} of projective spaces \citep{gortzAlgebraicGeometrySchemes2020} which play a prominent role in many areas of mathematics and physics, such as \emph{independence models} in algebraic statistics \citep{drtonLecturesAlgebraicStatistics2009} and \emph{entanglement} in quantum mechanics \citep{Bengtsson:2017aa}.
Based on this Ansatz, we develop a geometric perspective and formalism to study the relationship between replicator dynamics of multiple populations and multi-game replicator dynamics. In particular, we demonstrate that the multi-game dynamics of \cite{Hashimoto:2006} share a generic payoff structure with multi-population games. 

Our work further constitutes a formal reduction of multi-population dynamics to -- a much higher-dimensional -- single-population dynamics, which is helpful for theoretical analysis. We demonstrate this by transferring two results on the asymptotic behavior of replicator dynamics from the single-population to the multi-population setting. 

Concerning applications, our work aims to provide insight into the structure of multi-population and multi-game dynamics, along with a robust mathematical toolset for domain experts to analyze complex systems. 
Indeed, there is a growing need for more powerful dynamical models in emerging applications. For instance, \cite{Venkateswaran:2019} argue for the use of generalized replicator dynamics to model interactions in nature -- considering multi-player interaction in a multi-game setting.

The present paper also extends our previous work on \emph{assignment flows} \citep{Astroem2017, Schnorr2019aa}. They are dynamical systems that leverage interaction along edges of a graph $\Graph$ to infer an assignment of class labels to the nodes of $\Graph$ from node-wise data.
Applications include structured prediction problems such as semantic image segmentation in supervised \citep{Sitenko:2021vu} and unsupervised scenarios \citep{Zisler:2020aa,Zern:2020ab}.

In \citep{Boll:2021vb}, we have shown that assignment flows can be seen as multi-population replicator dynamics and studied how payoff is transformed by embedding the state space into a single simplex of joint distributions. In the present work, we generalize this analysis to nonlinear payoff functions and provide a careful study of the involved manifolds.

We also highlight previous findings on assignment flows and their relevance to the evolutionary game theory community. In particular
\begin{itemize}
\item \cite{Zern:2020aa} present an exhaustive study of conditions under which certain assignment flows converge to integer assignments. These are states in which only a single played strategy remains in each population. 
\item \cite{Zeilmann:2020aa} have proposed a generically applicable framework for geometric numerical integration which scales to large replicator dynamics.
\item \cite{Huhnerbein:2021th, Zeilmann:2022ul} have studied methods of learning the parameters that generate replicator dynamics from data.
\end{itemize}

\subsection{Organization}
Section~\ref{sec:preliminaries} contextualizes assignment flows as multi-population replicator dynamics and establishes related notation. 
Section~\ref{sec:embedding} describes the proposed meta-simplex concept and related geometric notions as well as the embedding theorems~\ref{prop:isometric_embedding_T} and~\ref{theorem:af_embedding}, which constitute the main results of the present work.
Section~\ref{sec:multi_games} gives three examples of dynamics considered in prior work and establishes their relationship through the lens of the geometric embedding theorems.
Section~\ref{sec:tangent_parameterization} recapitulates a tangent space parameterization of replicator dynamics from the literature on assignment flows and studies it in the context of geometric embedding.
Section~\ref{sec:learning} highlights previous findings on parameter learning for assignment flows.
Section~\ref{sec:asymptotic_behavior} demonstrates how the proposed formal reduction of multi-population to single-population dynamics can be used as a tool for formal analysis of asymptotic behavior.
Section~\ref{sec:conclusion} gives an outlook on current assignment flow developments and concludes the paper.\\[1em]
The present work substantially extends the conference paper \citep{Boll:2021vb} in the following ways:
\begin{itemize}
	\item The submanifold of embedded multi-population states is identified as a generalized Wright manifold and its geometry is analyzed (Theorem~\ref{prop:isometric_embedding_T}).
	\item The embedding theorem of multi-population replicator dynamics is generalized to nonlinear payoff functions (Theorem~\ref{theorem:af_embedding}).
	\item Multi-game dynamics are studied as embedded multi-population replicator dynamics (Section~\ref{sec:multi_games}).
	\item Tangent space parameterization of replicator dynamics is studied in the context of geometric embedding (Theorem~\ref{theorem:af_tangent_embedding}).
\end{itemize}

\subsection{Basic Notation.}

For $k\in\N$ we use the shorthands $[k] := \{ 1, \ldots, k\} \subset \N$ and $\eins_k := (1, \ldots, 1)^\T \in \R^k$. Angle brackets $\la \cdot, \cdot\ra$ are used for both the standard inner product between vectors and the Frobenius inner product between matrices. The Kronecker product of matrices \citep{Graham:1981wj} is denoted by $A\otimes B$. 
Componentwise multiplication of vectors $x$ and $y$ is denoted by $x\diamond y$, and by $\frac{x}{y}$ the componentwise division of a vector $x$ by a strictly positive vector $y$. Likewise, logarithms and exponentials of vectors apply componentwise.
For vectors $x\in \R^c$, the expression $x \geq 0$ denotes $x_i \geq 0$ for all $i \in [c]$.

\section{Preliminaries}\label{sec:preliminaries}
\subsection{Fisher-Rao Geometry, Replicator Dynamics}

In matrix games, players from a large population engage in two-player interactions. For simplicity, we assume that each player chooses from a constant set of $c$ strategies. The payoff for a two-player interaction is then given by a $c\times c$ payoff matrix $B$.
If players change their strategy to imitate other players with more effective strategies, the overall distribution $p \in \SM_c$ of strategies in the population changes over time according to the well-known \emph{replicator dynamics}
\begin{equation}\label{eq:replicator_base}
	\dot p(t) = \ROS_{p(t)}[Bp(t)],\qquad p(0) = p_0\in\SM_c,
\end{equation}
where 
\begin{equation}\label{eq:def-mcS}
\SM_c = \{p\in \R^c\colon \la \eins_c, p\ra = 1,\; p > 0\}
\end{equation}
denotes the relative interior of the probability simplex with $c$ vertices and 
\begin{equation}\label{eq:replicator_def}
	\ROS_{p} = \Diag(p) - pp^\top
\end{equation}
is called the \emph{replicator operator}. We regard $\SM_c$ as a Riemannian manifold with trivial tangent bundle
\begin{equation}\label{eq:trivial_tangent_bundle}
	T\SM_c \cong \SM_c \times T_0\SM_c,\qquad 
	T_0\SM_c = \{v\in \R^c\colon \la \eins_c, v\ra = 0\}
\end{equation}
and equipped with the \emph{Fisher-Rao metric}
\begin{equation}\label{eq:fisher_rao_metric}
	\la\cdot,\cdot\ra_g\colon T_p\SM_c\times T_p\SM_c\to \R,\qquad 
	(u,v) \mapsto \Big\la \frac{u}{p}, v\Big\ra.
\end{equation}
The barycenter of $\SM_c$ is denoted $\BS = \frac{1}{c}\eins_c$.
Vectors in $\R^c$ are projected onto the tangent space $T_0\SM_c$ by the linear map
\begin{equation}\label{eq:projection_def}
	\Pi_0\colon \R^c\to T_0\SM_c,\qquad v\mapsto \Pi_0v = v - \frac{1}{c}\la \eins_c, v\ra\eins_c.
\end{equation}
The manifold $\SM_c$ has dimension $c-1$. Two coordinate charts are particularly relevant to the following discussion. A point $p\in \SM_c\subseteq \R^c$ has $m$-coordinates $\mu$ with
\begin{equation}\label{eq:m_coordinates}
	p = (\mu, 1-\la \eins_{c-1}, \mu\ra),\qquad \mu\in\R^{c-1},\quad
\mu > 0,\quad
\la \mu, \eins_{c-1}\ra < 1
\end{equation}
and $e$-coordinates $\theta$ with
\begin{equation}\label{eq:e_coordinates}
	p = \frac{1}{Z(\theta)}\exp\bpm \theta \\ -\la \theta, \eins_{c-1}\ra\epm,\qquad \theta\in\R^{c-1},
\end{equation}
where $Z(\theta)$ normalizes the vector on the right-hand side such that $p\in\SM_c$, as defined by \eqref{eq:def-mcS}. The $e$-coordinates $\theta$ are unconstrained and define a \emph{global} chart for $\SM_c$.

For general references to Riemannian geometry, we refer to \citep{lee2018Riemannian, jost2017riemannian}. For references to \emph{information geometry} which underlies the above definitions, see \citep{amari2007methods, ay2017information}. 

The mean payoff in a population with state $p$ is $\la p, B p\ra$. In particular, if $B$ is symmetric $B p = \frac{1}{2}\partial_{p}\la p, B p\ra$, then it is well known that \eqref{eq:replicator_base} is the Riemannian gradient ascent flow of mean payoff with respect to the metric \eqref{eq:fisher_rao_metric}.
With an eye toward numerical computation, a useful object is the \emph{lifting map}
\begin{equation}\label{eq:lifting_map_def}
	\exp_p\colon T_0\SM_c\to \SM_c,\qquad
	\exp_p(v) = \frac{p\diamond\exp(v)}{\la p, \exp(v)\ra}.
\end{equation}
It can be shown that 
\begin{equation}\label{eq:exp-Pi}
\exp_{p} = \exp_{p}\circ \,\Pi_{0},
\end{equation}
with the projection $\Pi_{0}$ given by \eqref{eq:projection_def}, such that $\exp_{p}$ is well-defined on $\R^{c}$. Furthermore, the mapping 
\eqref{eq:lifting_map_def} is a first-order approximation to the Levi-Civita geodesics on $\SM_c$ \citep{Astroem2017}. It is also closely related to the $e$-geodesics of information geometry \citep{amari2007methods}.

\subsection{Data Labeling and Assignment Flows}\label{sec:af_prelim}
We have studied dynamics similar to \eqref{eq:replicator_base} for \emph{data labeling} on graphs. Given a graph $\Graph = (\Nodes, \Edges)$ and data on each node, the task is to infer node-wise classes. For example, in image segmentation, the graph may be a grid graph of image pixels, and pixel-wise data lives in some feature space, most basically a color space. 
Another example is node-wise classification of citation graphs such as \citep{Bollacker:1998}. Here, nodes are academic papers and edges between them denote citations. The task is to classify the topic of papers from node-wise features and citations.

In each case, graph connectivity is crucial information and a natural approach is to facilitate interaction along graph edges to inform the labeling process.
We further abstract from the raw feature space of given data by lifting to a probability simplex $\SM_c$ on each node. The resulting state lives in the product manifold 
\begin{equation}\label{eq:assignment_manifold}
	\WM = \SM_c\times \cdots \times \SM_c
\end{equation}
containing $n = |\Nodes|$ copies of $\SM_c$ which we call \emph{assignment manifold}. For $S\in \WM$, let $\ROW_S$ denote the operator which applies $\ROS_{S_i}$ on each node $i\in [n]$. Similarly, let
\begin{equation}\label{eq:def-exp-SV}
\exp_S(V),\qquad
V\in T_0\WM = (T_0\SM_c)^n
\end{equation}
denote the map which applies \eqref{eq:lifting_map_def} separately on each node, whose domain can be extended
\begin{equation}
\exp_{S}\colon\R^{n\times c}\to\mc{W}.
\end{equation}
due to \eqref{eq:exp-Pi}.
We will still call these objects \emph{replicator operator} and \emph{lifting map}, respectively. 
Likewise, the projection
\begin{equation}\label{eq:projection_Pi0}
	\Pi_0\colon\R^{n\times c}\to T_{0}\mc{W}
\end{equation}
applies separately the mapping \eqref{eq:projection_def} on each node and the barycenter of $\WM$ reads $\BW = \frac{1}{c}\eins_{n\times c}$.

By defining a payoff function
\begin{equation}\label{eq:payoff_def_W}
	F\colon \WM\to \R^{n\times c}
\end{equation}
which has the state $S_i, S_j \in \SM_c$ of nodes $i,j \in [n]$ interact exactly if $ij\in \Edges$, we have found a natural inference dynamic on $\Graph$ given by
\begin{equation}\label{eq:assignment_flow_def}
	\dot S(t) = \ROW_{S(t)}\big[F\big(S(t)\big)\big],\qquad S(0) = S_0,
\end{equation}
whose solution is called \emph{assignment flow}. 
For labeling, payoff functions are designed such that the state $S$ is driven towards an extremal point of the set $\WM$. These states unambiguously associate each node with a single class. From a game-theoretical perspective, the extremal points of $\WM$ are states in which only a single strategy is played in each population.
A more detailed overview can be found in the original work \citep{Astroem2017} and the survey \citep{Schnorr2019aa}.

\section{Embedding the Assignment Manifold}\label{sec:embedding}

In our previous work \citep{Boll:2021vb}, we showed that assignment flows can be seen as multi-population replicator dynamics. Furthermore, we introduced a preliminary formalism for embedding the state space of multi-population dynamics into a \emph{single}, much higher-dimensional \emph{meta-simplex} of joint distributions. 
Assuming again the data labeling perspective introduced in Section~\ref{sec:af_prelim}, one may enumerate all $c^n$ possible assignments of $c$ classes to $n$ graph nodes. 
This enumeration represents data labeling as a single decision between 
\begin{equation}
N = c^n
\end{equation}
alternatives which we view as pure strategies of a \emph{single} population game on the meta-simplex $\Sn$.

Here, we describe a refined version of the embedding formalism as well as several additional results, generalizing and expanding our earlier findings.
Note that the proposed meta-simplex $\Sn$ is not to be confused with the meta-simplex concept proposed by \cite{Argasinski:2006}. The latter explicitly considers the relative size of populations and has much lower dimension.

To simplify notation, we assume that agents of each population have the same number $c$ of available pure strategies. However, the following results remain valid in more general scenarios of variable strategy sets.
In addition, we index entries of vectors $\p \in \Sn$ by multi-indices $[c]^n$ as opposed to integer indices in $[c^n]$ to improve readability. The component $\gamma_{i}$ of a multi-index $\gamma \in [c]^{n}$ indexes a label $\gamma_{i}\in [c]$ at vertex $i\in [n]$.

We consider the following maps, defined componentwise by
\begin{subequations}\label{eq:central_maps}
\begin{align}
	T\colon&\mc W\to \imT\subseteq\Sn, &T(W)_\gamma :=& \prod_{i\in [n]} W_{i,\gamma_i} &&\text{for all } \gamma \in [c]^n\label{eq:T_def}\\
	Q\colon& \R^{n \times c}\to \R^N, &Q(X)_\gamma :=& \sum_{i\in [n]} X_{i,\gamma_i}&&\text{for all } \gamma \in [c]^n\label{eq:def_Q}\\
	\Margin \colon& \R^N \to \R^{n\times c}, &\Margin(x)_{ij} :=& \sum_{\gamma\in [c]^n\;:\; \gamma_i = j} x_\gamma&&\text{for all } (i, j) \in [n] \times [c].\label{eq:def_Margin}
\end{align}
\end{subequations}
The particular choice of these maps will be justified by laying out several compatibility properties which intricately link them to each other and to the geometries of $\WM$ and $\Sn$.
Specifically, 
\begin{itemize}
\item
$T$ realizes the concept of enumerating labelings in the sense that the extremal points of $\ol{\WM}$ are bijectively mapped to the extremal points of $\ol{\Sn}$. 
\item 
The restriction of $M$ to $\mc{T}$ inverts $T$ by computing node-wise marginals.
We choose the larger domain $\R^{N}$ for $M$ such that it becomes the adjoint mapping of $Q$ (ref Lemma~\ref{lem:Q_adjoint}).
\end{itemize}
\begin{theorem}[\textbf{assignment manifold embedding}]\label{prop:isometric_embedding_T}
The map $T\colon \WM\to\imT\subseteq \Sn$ is an isometric embedding of $\WM$ equipped with the product Fisher-Rao geometry, into $\Sn$ equipped with the Fisher-Rao geometry. On its image $T(\WM) =: \imT\subseteq \Sn$, the inverse is given by marginalization 
\begin{equation}\label{eq:T_inv}
	\Margin|_\imT = T^{-1} \colon \imT \to \WM.
\end{equation}
\end{theorem}
\begin{proof}
Section \ref{proof:prop:isometric_embedding_T}.
\end{proof}

In view of the expression \eqref{eq:T_def}, it is clear that $\imT$ is precisely the set of rank-1 tensors in $\Sn \subseteq \R^N \cong (\R^c)^n$. In addition, we have the following interpretation of points $T(W) \in \imT$ within the simplex of joint distributions $\Sn$.

\vspace{0.25cm}
\begin{proposition}[\textbf{maximum entropy property}]\label{prop:max_entropy}
For every $W\in\WM$, the distribution $T(W)\in\Sn$ has maximum entropy among all $p\in \Sn$ subject to the marginal constraint $\Margin p = W$, with $\Margin$ given by \eqref{eq:def_Margin}.
\end{proposition}
\begin{proof}
Section \ref{proof:prop:max_entropy}.
\end{proof}

In general, each collection of marginal distributions $S\in\WM$ has (infinitely) many possible joint distributions. Proposition~\ref{prop:max_entropy} shows that $T$ precisely selects the least informative one among them. This situation is illustrated in Figure~\ref{fig:marginalization}.
\begin{figure}[ht]
\centering
\begin{subfigure}[b]{0.45\textwidth}
	\centering
	\includegraphics[width=\textwidth]{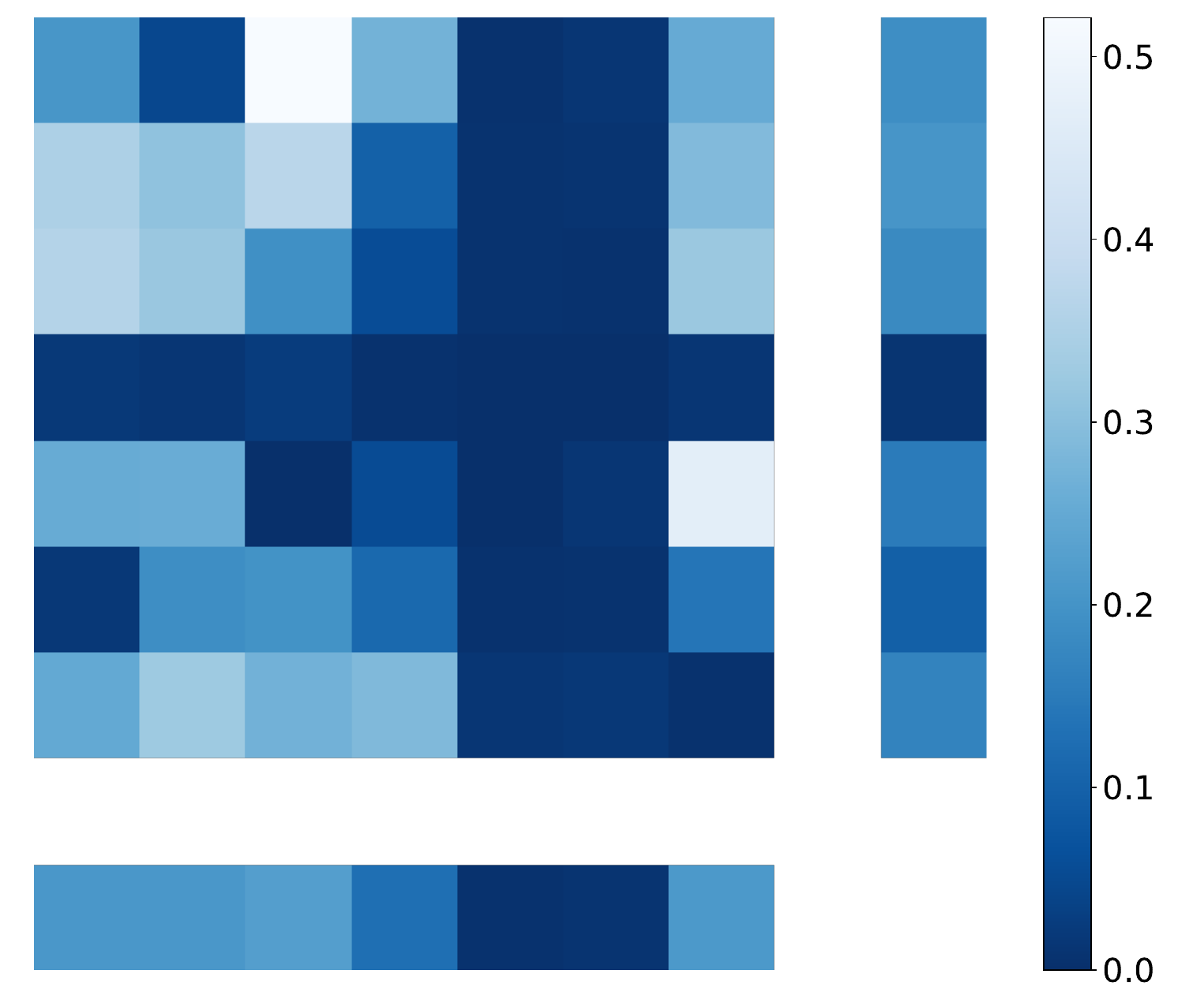}
	\caption{A randomly generated joint distribution of $S_1$ and $S_2$.}
	\label{fig:marginalization_random}
\end{subfigure}
\hfill
\begin{subfigure}[b]{0.45\textwidth}
	\centering
	\includegraphics[width=\textwidth]{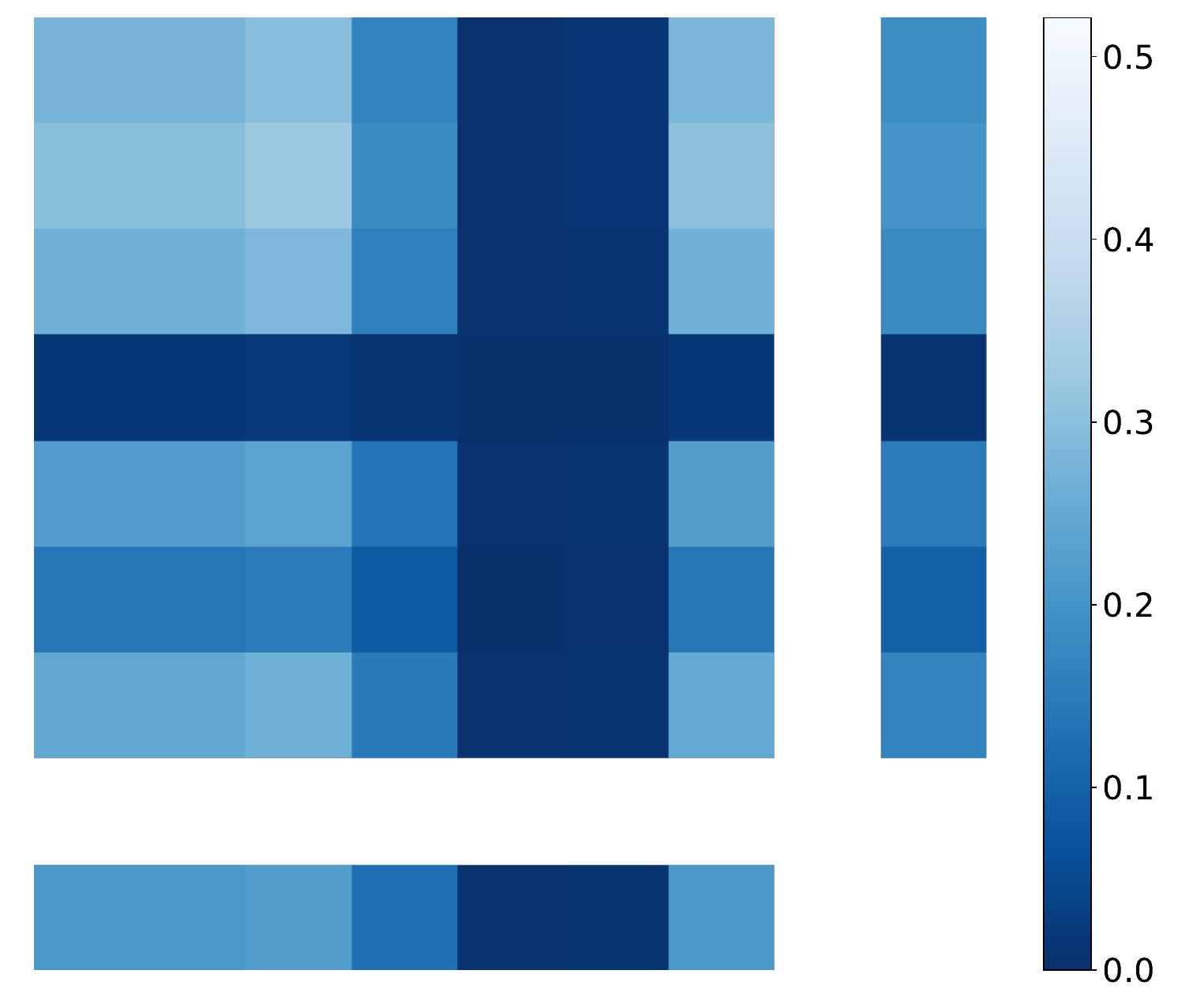}
	\caption{The maximum-entropy joint distribution $T(S)$ of $S_1$ and $S_2$.}
	\label{fig:marginalization_T}
\end{subfigure}
\caption{Marginals distributions $S = (S_1, S_2)$ and two possible conforming joint distributions. Joint distribution values are scaled by a factor of $c$ for visual clarity.}\label{fig:marginalization}
\end{figure}

Theorem \ref{prop:isometric_embedding_T} expresses an intricate relationship between the product Fisher-Rao geometry of $\WM$ and the Fisher-Rao geometry of $\Sn$. A similar compatibility is found between the lifting map \eqref{eq:lifting_map_def} on $\WM$ and its analog on $\Sn$.

\vspace{0.25cm}
\begin{lemma}[\textbf{Lifting Map Lemma}]\label{lem:lifting_map}
	Let $S\in \WM$ and $V\in \R^{n\times c}$. Then the mappings $T, Q$ given by \eqref{eq:central_maps} satisfy
	\begin{equation}\label{eq:lifting_map_delta}
		T\big(\exp_{S}(V)\big) = \exp_{T(S)}\big(Q(V)\big),
	\end{equation}
where $\exp_{S}$ on the left is given by \eqref{eq:def-exp-SV} and $\exp_{T(S)}$ on the right naturally extends the mapping \eqref{eq:lifting_map_def}.
\end{lemma}
\begin{proof}
Section \ref{proof:lem:lifting_map}.
\end{proof}

We will also frequently use the following useful identity connecting $Q$ to marginalization.
\begin{lemma}[\textbf{$Q$ Adjoint Lemma}]\label{lem:Q_adjoint}
$\Margin$ and $Q$ given by given by \eqref{eq:central_maps} are adjoint linear maps with respect to the standard inner product, i.e. for each $p\in \R^N$ and each $V\in \R^{n\times c}$ it holds that
\begin{equation}\label{eq:Q_adjoint}
  \la p, Q(V)\ra = \la Mp, V\ra.
\end{equation}
\end{lemma}
\begin{proof}
  Let $\p\in \R^{N}$ and $V\in\R^{n\times c}$, then
  \begin{subequations}\label{eq:inner_product_proof_off}
  \begin{align}
    \la \p, Q(V)\ra &= \sum_{\gamma\in [c]^n} \p_\gamma Q(V)_\gamma
    = \sum_{\gamma\in [c]^n} \p_\gamma \sum_{l\in [n]} V_{l,\gamma_l}
    = \sum_{l\in [n]}\sum_{j\in [c]}\sum_{\gamma_l = j} \p_\gamma V_{l,\gamma_l}\label{eq:inner_product_proof_off3}\\
    &= \sum_{l\in [n]}\sum_{j\in [c]} V_{l,j}\sum_{\gamma_l = j} \p_\gamma
    \stackrel{\eqref{eq:def_Margin}}{=} \sum_{l\in [n]}\sum_{j\in [c]} V_{l,j} (\Margin \p)_{l,j}\label{eq:inner_product_proof_off5}\\
    &= \la \Margin \p, V\ra.
    \label{eq:inner_product_proof_off6}
  \end{align}
  \end{subequations}
\end{proof}

Our main result stated next is that the embedding $T\colon \WM\to \Sn$ maps \textit{multi}-population replicator dynamics on $\WM$ to \textit{single}-population replicator dynamics on $\Sn$ by a transformation of payoff functions. This generalizes our earlier finding \citep{Boll:2021vb} to arbitrary nonlinear payoffs.

\vspace{0.25cm}
\begin{theorem}[\textbf{Multi-Population Embedding Theorem}]\label{theorem:af_embedding}
For any payoff function $F\colon \WM\to \R^{n\times c}$, the multi-population replicator dynamics
\begin{equation}\label{eq:general_af}
	\dot W = \ROW_W[F(W)], \quad W(0) = W_0
\end{equation}
on $\WM$ is pushed forward by $T$ to the replicator dynamics
\begin{equation}\label{eq:replicator_ms}
	\dot p(t) = \ROS_{p(t)}\wh F\big(p(t)\big), \quad p(0) = T(W_0), \qquad \wh F = Q \circ F \circ \Margin,
\end{equation}
on $\Sn$ and the map $T$ satisfies
\begin{equation}\label{eq:dT_ReplOp_main}
  dT|_W[\ROW_W[X]] = R_{T(W)} Q[X], \quad \text{for all } X \in \R^{n \times c} \text{ and } W\in \WM.
\end{equation}
\end{theorem}
\begin{proof}
Section \ref{proof:theorem:af_embedding}.
\end{proof}

Intuitively, the structure of $\wh F$ in \eqref{eq:replicator_ms} can be seen as follows.
The joint population state $\p\in \Sn$ is first marginalized and payoff $F(\Margin \p)$ is computed from the marginal multi-population state. Theorem~\ref{theorem:af_embedding} now shows that when multi-population state $W\in \WM$ is seen as factorizing joint population state $\p\in \Sn$ according to $\p = T(W)$, then the payoff gained in state $W$ is transformed by $Q$ to induce replicator dynamics of the joint population state.

In the following, leading examples will be matrix games, i.e. linear payoff functions that model two-player interactions. Note, however, that Theorem~\ref{theorem:af_embedding} applies to arbitrary nonlinear payoff functions, including multi-player interactions.

\begin{figure}[ht]
\centering
\includegraphics[width=.6\textwidth]{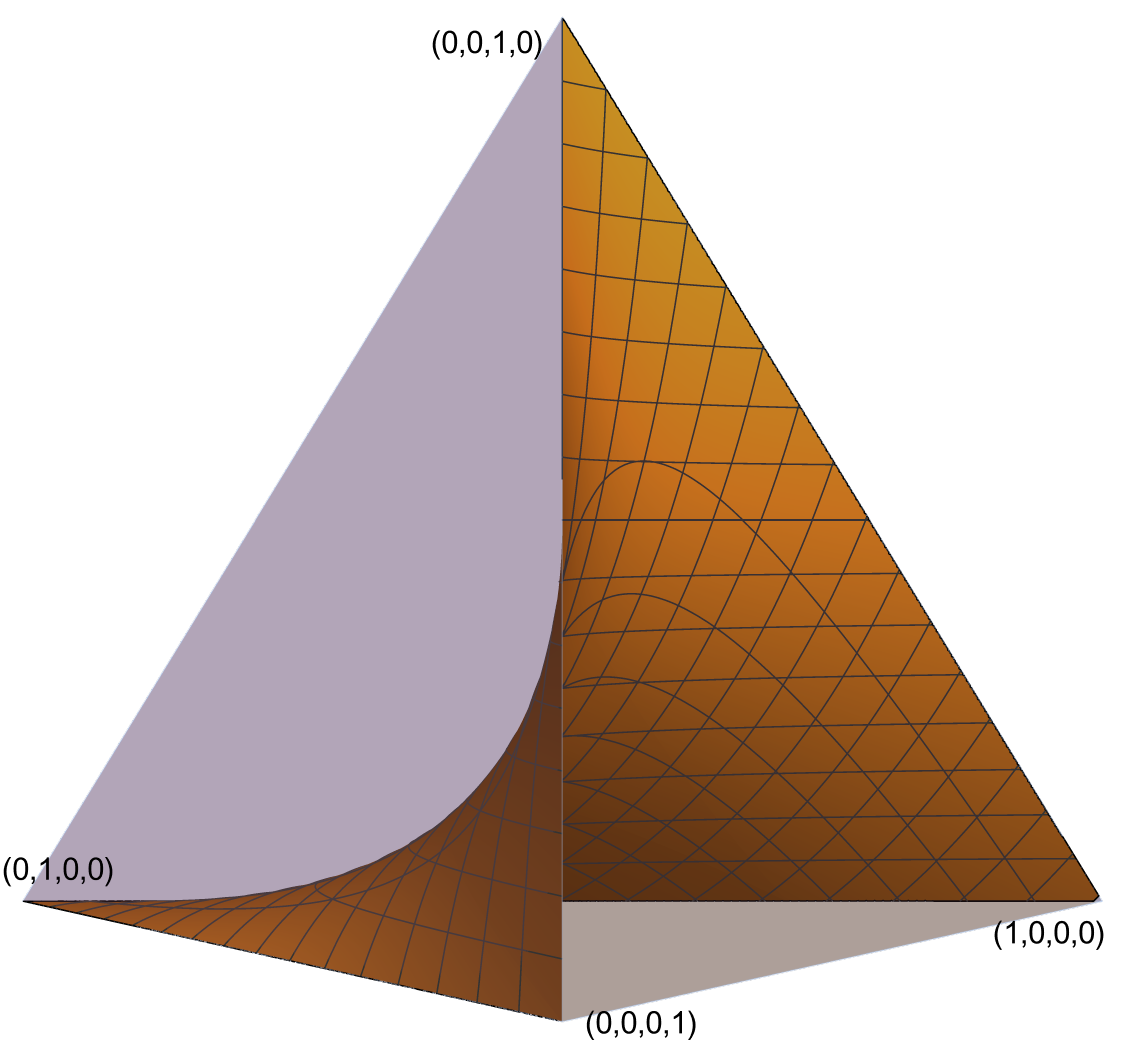}
\caption{The embedded submanifold $\imT\subseteq \Sn,\, N=4$. For two marginal distributions, this is known as the \emph{Wright manifold} \citep[Section 18.8]{Hofbauer:1998}, \citep{Chamberland:2000}.}\label{fig:wright_manifold}
\end{figure}

\section{Multiple Populations and Multiple Games}\label{sec:multi_games}
Because both $\Margin$ and $Q$ are linear operators, generalized matrix games on multiple populations reduce to simple matrix games of the joint population state exactly if the payoff $F$ is a linear function of the multi-population state. Here, we give two examples of multi-population games, one from the assignment flow literature and one from game theory. To this end, denote by 
\begin{equation}\label{eq:s-vec-S}
s := \vvec_{\text{row}}(S)\in \R^{nc}
\end{equation}
a vectorized multi-population state which contains all entries of $W\in \WM\subseteq \R^{n\times c}$ stacked row-wise. Table \ref{tab:payoff_structure} summarizes the scenarios discussed in the following.

\begin{table}[h]
\centering
\caption{Structure of payoff \eqref{eq:arbitrary_payoff_matrix} for simple instances of different games.}\label{tab:payoff_structure}
\begin{tabular}{@{}llll@{}}
\toprule
           & S-Flow                 & EGN                & Multi-Game        \\ \midrule
Payoff     & $\ol{A} = \Omega \otimes \II_c$ & $ \ol{A} = \Omega \otimes B$ & $\ol{A} = \II_n \otimes B$ \\ \bottomrule
\end{tabular}
\end{table}

\emph{S-flows} \citep{Savarino2019ab} define payoff by averaging the state $S$ according to a weighted graph adjacency matrix $\Omega \in \R^{n\times n}$. The resulting assignment flow with vectorized state \eqref{eq:s-vec-S} reads
% \begin{equation}\label{eq:theo_S_field}
% 	\dot W(t) = \ROW_{W(t)} [\Omega W(t)],\qquad W(0) = W_0
% \end{equation}
% and vectorization of its multi-population state gives
\begin{equation}\label{eq:vec_S_field}
	\dot s(t) = \ROW_{s(t)} [(\Omega \otimes \II_c) s(t)],\qquad s(0) = s_0.
\end{equation}
This dynamical system promotes similarity of adjacent populations. Depending on the initialization $s_0$, `pockets of consensus' are formed. It has also been shown that these dynamics converge to extremal points of $\WM$ for almost all initializations under weak conditions \citep{Zern:2020aa}.\\[1em]
\emph{Evolutionary Games on Networks (EGN)} \citep{Madeo:2014, Iacobelli:2016} are dynamics which generalize \eqref{eq:vec_S_field} by incorporating payoff matrices for games played between players of adjacent populations. In the simplest case, all such games have a constant payoff matrix $B\in \R^{c\times c}$. Then, the multi-population replicator dynamics of EGN read
\begin{equation}\label{eq:vec_egn_field}
	\dot s(t) = \ROW_{s(t)} [(\Omega \otimes B^\top) s(t)],\qquad s(0) = s_0.
\end{equation}
Both \eqref{eq:vec_S_field} and \eqref{eq:vec_egn_field} have a linear (in the vectorized state $s$) payoff function. Let 
\begin{equation}\label{eq:arbitrary_payoff_matrix}
	\ol{A}\in \R^{nc\times nc}
\end{equation}
be an arbitrary payoff matrix for the vectorized state. Then by Lemma~\ref{lem:Q_adjoint} and Theorem~\ref{theorem:af_embedding}, the embedded dynamics in $\imT\subseteq \Sn$ read
\begin{equation}\label{eq:emedded_lin_dyn}
	\dot p(t) = R_{p(t)}[Q \ol{A} Q^\top p(t)],\qquad p(0) = T(s_0).
\end{equation}
The \emph{multi-game dynamics} of \cite{Hashimoto:2006} can also be written as a matrix game in $\Sn$. Given matrices $A^{(i)}\in \R^{c\times c}$, $i\in [n]$, it reads
\begin{equation}\label{eq:hashimoto-replicator}
	\dot p(t) = R_{p(t)}[A p(t)],\qquad p(0) = p_0,\qquad A_{\alpha, \beta} = \sum_{i\in [n]} A^{(i)}_{\alpha_i, \beta_i}.
\end{equation}
The structure of this payoff matrix has a natural shape within our formalism, too.

\vspace{0.25cm}
\begin{lemma}\label{lem:hashimoto_blockdiag}
The payoff matrix in \eqref{eq:hashimoto-replicator} can be written as $A = Q \ol{A} Q^\top$ where $\ol{A}$ denotes the block diagonal matrix with diagonal blocks $A^{(i)}$.
\end{lemma}
\begin{proof}
\begin{subequations}
\begin{align}
(Q\ol{A}Q^\top)_{\alpha,\beta}
	&= \la e_\alpha, Q\ol{A}Q^\top e_\beta\ra
	= \la Q^\top e_\alpha, \ol{A}Q^\top e_\beta\ra\\
	&= \sum_{i\in [n]} \la e_{\alpha_i}, A^{(i)} e_{\beta_i}\ra
	= \sum_{i\in [n]} A^{(i)}_{\alpha_i,\beta_i}
\end{align}
\end{subequations}
\end{proof}
In particular, if all single-game payoff submatrices are the same $A^{(i)} = B \in \R^{c\times c}$, then multi-game dynamics have payoff $\ol{A} = \II_n \otimes B$.

It was shown by \cite{Hashimoto:2006} that the multi-game dynamics \eqref{eq:hashimoto-replicator} do not generally decompose into individual single-game dynamics, unless the initialization is on the \emph{Wright manifold} (see Figure~\ref{fig:wright_manifold}). The set $\imT\subseteq \Sn$ defined by \eqref{eq:T_def} is a generalization of the Wright manifold for $n > 2$ and Theorem~\ref{theorem:af_embedding} generalizes the decomposition of multi-game dynamics to more than two populations. For $p(0)\in \imT$, the dynamics \eqref{eq:hashimoto-replicator} is the embedded dynamics of
\begin{equation}\label{eq:hashimoto-decomposed}
	\dot s(t) = \ROW_{s(t)}[\ol{A}s(t)],\qquad s(0) = \Margin p(0)
\end{equation}
by Lemma~\ref{lem:hashimoto_blockdiag} and Theorem~\ref{theorem:af_embedding}. Since $\ol{A}$ is block diagonal, \eqref{eq:hashimoto-decomposed} is a collection of non-interacting single-game replicator dynamics
\begin{equation}\label{eq:hashimoto-decomposed-i}
	\dot W_i(t) = R_{W_i(t)}[A^{(i)}W_i(t)],\qquad W_i(0) = (\Margin p(0))_i,\qquad i\in [n]
\end{equation}
in accordance with the findings of \cite{Hashimoto:2006} for the specific case $n = 2$.

\section{Tangent Space Parameterization}\label{sec:tangent_parameterization}

Multi-population replicator dynamics evolve in the \emph{curved} space $\WM$ and the usual parameterization in $m$-coordinates of information geometry is subject to simplex constraints on the state.
With an eye toward numerical integration, it is desirable to instead parameterize replicator dynamics in a \emph{flat} and \emph{unconstrained} vector space. This was done in \citep{Zeilmann:2020aa} using Lie group methods.

\vspace{0.25cm}
\begin{theorem}[Proposition~3.1 in \citep{Zeilmann:2020aa}]\label{theorem:af_tangent_param}
The solution for multi-population replicator dynamics
\begin{equation}\label{eq:general_af_before_param}
	\dot W(t) = \ROW_{W(t)}[F(W(t))], \quad W(0) = W_0
\end{equation}
in $\WM$ admits the parameterization
\begin{subequations}\label{eq:tangent_param}
\begin{align}
	W(t) &= \exp_{\BW}(V(t))\\
	\dot V(t) &= \Pi_0 F(\exp_{\BW}(V(t))),\qquad V(0) = \Pi_0 \log W_0
\end{align}
\end{subequations}
in the tangent space $V(t)\in T_0\WM$.
\end{theorem}

With regard to the Embedding Theorem~\ref{theorem:af_embedding}, it turns out that while $T$ maps assignment matrices $W\in \WM$ to joint states $p\in\Sn$, $Q$ assumes a corresponding role for tangent vectors in $T_0\WM$.

\vspace{0.25cm}
\begin{theorem}[\textbf{Tangent Space Embedding Theorem}]\label{theorem:af_tangent_embedding}
The multi-population tangent space replicator dynamics
\begin{equation}\label{eq:general_af_tangent}
	\dot V = \Pi_0 F\big(\exp_{\BW}(V)\big), \quad V(0) = V_0
\end{equation}
on $T_0\WM$ is pushed forward by $Q$ to the tangent space replicator dynamics
\begin{equation}\label{eq:replicator_ms_tangent}
	\dot U = \Pi_0 \wh F\big(\exp_{\eins_N}(U)\big), \quad U(0) = Q(V_0), \qquad \wh F = Q \circ F \circ \Margin
\end{equation}
on $T_0\Sn$.
\end{theorem}
\begin{proof}
Denoting $U = QV$ and using the lifting map (Lemma~\ref{lem:lifting_map}), we directly compute
\begin{subequations}
\begin{align}
	\dot U &= Q\dot V = Q\Pi_0 F\big(\exp_{\BW}(V)\big)\\
		&= \Pi_0 QF\big(\exp_{\BW}(V)\big)
		&&\text{ by Lemma~\ref{lem:Q_proj_commute}}\\
		&= \Pi_0 QF\big((\Margin\circ T)\big(\exp_{\BW}(V)\big)\big)
		&&\text{ by \eqref{eq:T_inv}}\\
		&= \Pi_0 QF\big(\Margin \exp_{\eins_N}(QV)\big)		&&\text{ by Lemma~\ref{lem:lifting_map}}\\
		&= \Pi_0 QF\big(\Margin \exp_{\eins_N}(U)\big)\\
		&= \Pi_0 \wh F\big(\exp_{\eins_N}(U)\big).
\end{align}
\end{subequations}
\end{proof}

Pushforward via $Q$ thus preserves the shape of \eqref{eq:general_af_tangent} up to the same fitness function transformation $\wh F = Q \circ F \circ T^{-1}$ from Theorem~\ref{theorem:af_embedding}. 

The set $\mimg Q \subseteq T_0\Sn$ contains exactly those tangent vectors corresponding to assignments $\imT \subseteq \Sn$ via lifting, because $T(W) = T(\exp_{\BW})(V) = \exp_{\eins_N}(QV)$ for any $W\in\WM$ and $V = \exp_{\BW}^{-1}(W)$ by Lemma~\ref{lem:lifting_map}.
In particular, the set $\mimg Q$ in which $U$ evolves, is a linear subspace of $T_0\WM$. This is a reason to study the tangent space flow \eqref{eq:replicator_ms_tangent} rather than the corresponding replicator dynamics if applicable, because $\imT\subseteq \Sn$ is the (curved) set of rank-1 tensors in $\Sn$.

\section{Learning Replicator Dynamics from Data}\label{sec:learning}

\begin{figure}[ht]
\centering
\includegraphics[width=\textwidth]{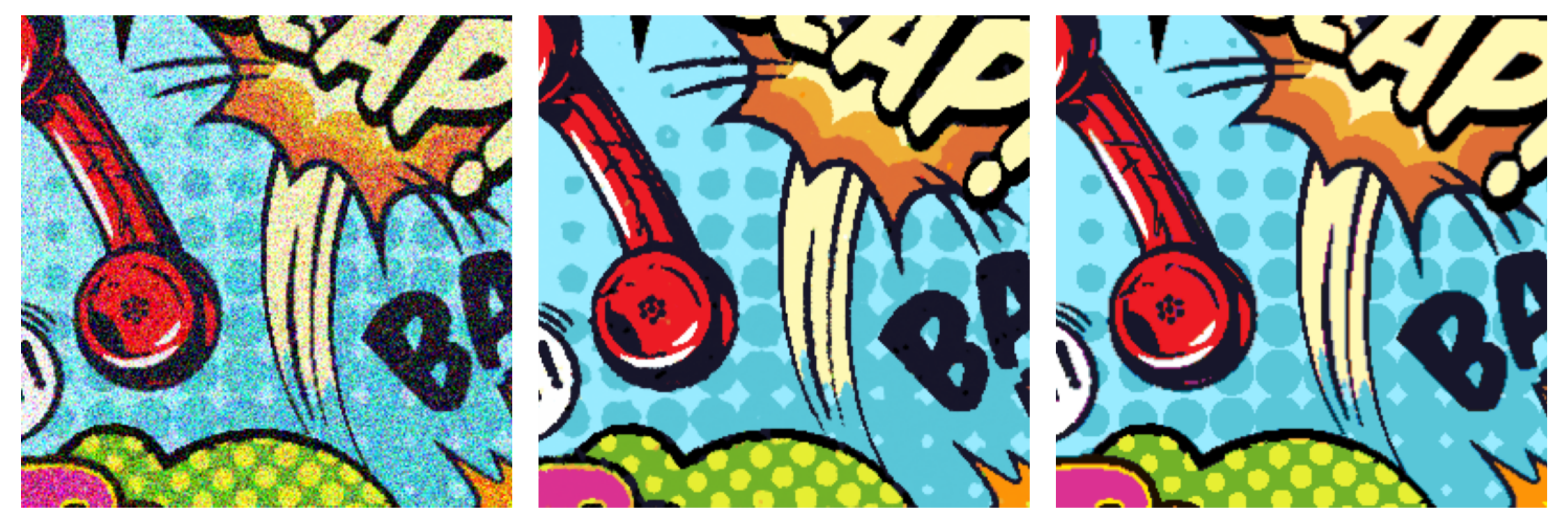}
\caption{\emph{Left}: Noisy input assignment of $c=47$ colors to the pixels of an image. \emph{Center}: Limit of an EGN flow \eqref{eq:vec_egn_field} with \textit{learned} interaction in $3\times 3$ pixel neighborhoods. \emph{Right}: Ground truth noise-free color assignment.}\label{fig:cartoon_labeling}
\end{figure}

Several applications have been proposed for the replicator dynamics of Section~\ref{sec:multi_games} including as a model of human brain functioning \citep{Madeo:2017}, collective learning \citep{Sato:2003}, epileptic seizure onset detection \citep{Hamavar:2021}, task mapping \citep{Madeo:2020} and collective adaptation \citep{Sato:2005}.
Assignment flows have been applied recently to the segmentation of digitized volume data under layer ordering constraints \citep{Sitenko:2021vu} (which reflect prior knowledge about tissues and anatomical structure) as well as for unsupervised image labeling tasks, employing spatial regularization \citep{Zisler:2020aa,Zern:2020ab}.

This small sample of examples illustrates that replicator dynamics can act as powerful data models in diverse applications. In situations where only partial knowledge about the system is available, system parameters may also be learned from data. To this end, \cite{Huhnerbein:2020aa} have studied the use of \emph{adjoint integration} to compute the model sensitivity of assignment flows, i.e. the gradient of system state with respect to parameters generating the flow.

Suppose we integrate a general dynamical system generated by parameters $\mathfrak{p}$ and wish for the final state $v(T)$ to minimize some loss function $\mc{L}$. The parameter learning problem for a fixed time horizon $T > 0$ then reads
\begin{subequations}\label{eq:learning_problem}
\begin{align}
	\min_{\mathfrak{p}}\, \mc{L}(&v(T, \mathfrak{p}))\\
	\text{subject to}\quad \dot v(t) &= f(v(t), \mathfrak{p}, t),\quad t\in [0,T],\label{eq:learning_forward_ode}\\
		v(0) &= v_0,
\end{align}
\end{subequations}
and a central quantity of interest is the gradient $\partial_\mathfrak{p} \mc{L}(v(T, \mathfrak{p}))$. It can be approximated in a \emph{discretize-then-optimize} fashion by first choosing a discretization of the ODE \eqref{eq:learning_forward_ode} on $[0,T]$ and subsequently computing the gradient of the discrete scheme used for computing $\mc{L}(v(T, \mathfrak{p}))$. This approach is easy to implement by using automatic differentiation software \cite{Baydin:2018aa}. However, it entails a large memory footprint in practical applications because system state $v(t_i)$ needs to be saved for all discretization points.
To circumvent this issue, one may instead proceed in an \emph{optimize-then-discretize} fashion as follows.

\vspace{0.25cm}
\begin{theorem}[Theorem~6 of \citep{Huhnerbein:2020aa}]\label{theorem:adjoint_ode}
The gradient of \eqref{eq:learning_problem} is given by
\begin{equation}\label{eq:adjoint_gradient}
	\partial_\mathfrak{p} \mc{L}(v(T, \mathfrak{p})) = \int_0^T \dd_\mathfrak{p} f (v(t), \mathfrak{p}, t)^\top \lambda(t) \dd t
\end{equation}
where $\dd_\mathfrak{p} f$ denotes the differential of $f$ with respect to $\mathfrak{p}$ and $x(t)$ and $\lambda(t)$ solve
\begin{subequations}\label{eq:adjoint_ode}
\begin{align}
	\dot v(t) &= f(v(t), \mathfrak{p}, t),\qquad v(0) = v_0,\\
	\dot \lambda(t) &= -\dd_v f(v(t), \mathfrak{p}, t)^\top \lambda(t),\qquad \lambda(T) = \partial \mc{L}(v(T)).
\end{align}
\end{subequations}
\end{theorem}
By choosing a quadrature for the integral \eqref{eq:adjoint_gradient}, Theorem~\ref{theorem:adjoint_ode} allows to compute the desired gradient without the need to save system state at all discretization points. Moreover, it has been shown \citep{Huhnerbein:2020aa, Sanz-Serna:2016} that for particular \emph{symplectic} integrators, discretization commutes with optimization, i.e. both orders of operation yield the same gradient.

Since the tangent space parameterization \eqref{eq:tangent_param} evolves on the unconstrained flat space $\TW$, Theorem~\ref{theorem:adjoint_ode} is directly applicable to it. 
An image labeling example is shown in Figure~\ref{fig:cartoon_labeling}. Here, the data is modeled by EGN dynamics \eqref{eq:vec_egn_field} with graph adjacency matrix $\Omega$ representing pixel neighborhoods ($3\times 3$). Starting from a noisy, high-entropy assignment of pixels to color prototypes ($c = 47$), the goal is to learn a label interaction matrix $B\in \R^{c\times c}$ such that EGN dynamics \eqref{eq:vec_egn_field} drive the state to a given noise-free assignment after the fixed integration time $T = 15$. We initialized $B$ as identity matrix and performed $100$ steps of the Adam optimizer to minimize cross-entropy between the ground truth assignment and the assignment state reached by EGN dynamics. This training procedure is highly scalable -- for the $256\times 256$ pixel image in Figure~\ref{fig:cartoon_labeling}, training takes less than a minute on a laptop computer and requires around 1.3GB of vRAM.

\section{Asymptotic Behavior}\label{sec:asymptotic_behavior}
A central topic in population dynamics is the study of how the properties of the underlying game characterized by the payoff function relate to steady states of the dynamical model. In this section, we describe how 
\begin{itemize}
\item
\emph{Nash equilibria (NE)} and 
\item
\emph{Evolutionarily stable states (ESS)} 
\end{itemize}
of multi-population games and their replicator dynamics behave under the embedding \eqref{eq:T_def}.  
Nash equilibria for multi-population games are population states at which no agent (in any population) has payoff to gain from unilaterally switching strategies.

\vspace{0.25cm}
\begin{definition}[\textbf{Nash Equilibrium}]\label{def:nash}
Let $\ol{\WM}$, the closure of $\WM$ be the set of multi-population states ($n$ populations, $c$ strategies) and let $F\colon \ol{\WM}\to \R^{n\times c}$ be the payoff for a multi-population game. The set of Nash equilibria of $F$ is defined as
\begin{equation}\label{eq:def_nash}
    \NE (F) = \{W\in\ol{\WM}\;|\; \forall i\in [n],\, \forall j\in \supp (W_i),\, \forall k\in [c]\,\colon F(W)_{i,j} \geq F(W)_{i,k}\}
\end{equation}
\end{definition}

Definition~\ref{def:nash} naturally extends the classic notion of Nash equilibrium to multi-population games.
Nash equilibria are preserved if the multi-population game is embedded as specified by Theorem~\ref{theorem:af_embedding}.

\vspace{0.25cm}
\begin{theorem}[\textbf{Embedded Nash Equilibria}]\label{theorem:embedded_nash}
Let $F\colon \ol{\WM}\to \R^{n\times c}$ be a multi-population game on $\ol{\WM}$ and $\wh F = Q\circ F\circ \Margin$ be the related population game on $\Sn$. Then 
\begin{equation}\label{eq:embedded_nash}
T\big(\NE (F)\big) = \NE (\wh F) \cap \ol{\imT}.
\end{equation}
\end{theorem}
\begin{proof}
Let $W\in \NE(F)$ and let $\alpha\in\supp(T(W))$ be arbitrary. Then
\begin{subequations}
\begin{align}
    \wh F(T(W))_\alpha = (QF(W))_\alpha 
    &= \sum_{l\in[n]} F_{l,\alpha_l}(W) \\
    &\geq \sum_{l\in[n]} F_{l,\beta_l}(W) = \wh F(T(W))_\beta,\quad \forall\,\beta\in [c]^n,
\end{align}
\end{subequations}
because $\alpha_l \in \supp(W_l),\;\forall l\in[n]$, by Lemma~\ref{lem:embedded_support} and $W$ is a Nash equilibrium of $F$. This implies $T(\NE (F)) \subseteq \NE (\wh F) \cap \ol{\imT}$. Conversely, let $\p\in\NE(\wh F)\cap \imT$ have shape $\p=T(W)$ and let $\alpha_l \in \supp(W_l),\;\forall l\in[n]$. Then $\alpha\in\supp(\p)$ by Lemma~\ref{lem:embedded_support} and
\begin{equation}\label{eq:embedded_nash_proof1}
    \sum_{l\in[n]} F_{l,\alpha_l}(W) = \wh F(\p)_\alpha \geq \wh F(\p)_\beta = \sum_{l\in[n]} F_{l,\beta_l}(W),\quad \forall\,\beta\in [c]^n,
\end{equation}
because $\p$ is a Nash equilibrium. Choose $\beta\in[c]^n$ such that it matches $\alpha$ at all positions but $i\in[n]$. Then \eqref{eq:embedded_nash_proof1} implies $F_{i,\alpha_i}(W) \geq F_{i,\beta_i}(W)$ for arbitrary $\beta_i\in[c]$ which shows $\NE (\wh F) \cap \ol{\imT} \subseteq T(\NE (F))$.
\end{proof}

\begin{definition}[\textbf{Evolutionarily Stable State (ESS)}]\label{def:ess}
A multi-population state $W^\ast\in \WM$ is called an evolutionarily stable state (ESS) of a game $F\colon \WM\to \R^{n\times c}$, if there is an environment $U\subseteq \WM$ of $W^\ast$ such that
\begin{equation}\label{eq:ess_def}
    \la W - W^\ast, F(W)\ra < 0,\qquad \forall\, W\in U\setminus \{W^\ast\}.
\end{equation}
\end{definition}

This generalization of the classic ESS \citep{Smith:1973} to multi-population settings is called \emph{Taylor ESS} by \cite{Sandholm2010}. Within our embedding framework, an apparent reason recommends Definition~\ref{def:ess} over the weaker notion of \emph{monomorphic ESS} \citep{Cressman:1992}.

\vspace{0.25cm}
\begin{theorem}[\textbf{Embedded ESS}]\label{theorem:embedded_ess}
Let $F\colon \WM\to\R^{n\times c}$ be a multi-population game. Then $W^\ast$ is an ESS of $F$ exactly if there exists an environment $U\subseteq \imT$ of $T(W^\ast)$ such that
\begin{equation}\label{eq:embedded_ess}
    \la p - T(W^\ast), \wh{F}(p)\ra < 0,\qquad \forall\, p\in U\setminus \{T(W^\ast)\},
\end{equation}
where $\wh F = Q\circ F \circ \Margin$ denotes the embedded single-population game on $\Sn$ as specified by Theorem~\ref{theorem:af_embedding}.
\end{theorem}
\begin{proof}
Since $U\subseteq\imT$, we may write $p=T(W)$ for $W$ in an environment $\Margin(U)\subseteq \WM$ of $W^\ast$. \eqref{eq:embedded_ess} then reads
\begin{subequations}
\begin{align}
    \la T(W) - T(W^\ast), \wh{F}(T(W))\ra
    &= \la T(W) - T(W^\ast), QF(\Margin T(W))\ra\\
    &= \la \Margin(T(W) - T(W^\ast)), F(W)\ra\text{ (Lemma~\ref{lem:Q_adjoint})}\\
    &= \la W - W^\ast, F(W)\ra
\end{align}
\end{subequations}
and the last row is strictly smaller than $0$ for all $W\in \Margin(U)\setminus \{W^\ast\}$ exactly if $W^\ast$ is an ESS of $F$ according to Definition~\ref{def:ess}.
\end{proof}

One useful aspect of Theorem~\ref{theorem:af_embedding} is that it formally reduces multi-population replicator dynamics to single-population ones. This enables us to transfer analysis of e.g. asymptotic behavior from the single-population to the multi-population setting.
We first summarize standard results on the asymptotic behavior of replicator dynamics derived from a potential function and refer to \cite{Sandholm2010} for a comprehensive overview.

\vspace{0.25cm}
\begin{theorem}[\textbf{Replicators converge to NE}]\label{theorem:replicator_opt}
Let $\wh{J}\colon\mc{S}_c\to\R$ be a $C^1$ potential such that the induced payoff function $\wh{F} = \Pi_0\nabla \wh{J}$  is Lipschitz on $\mc{S}_c$. Then for any internal point $p_0\in\mc{S}_c$, the replicator dynamics
\begin{equation}\label{eq:replicator_opt}
    \dot W(t) = R_{p(t)}[\wh{F}(p)],\qquad p(0) = p_0
\end{equation}
converge to a Nash equilibrium.
\end{theorem}
\begin{proof}
Because $F$ is Lipschitz, the forward trajectories of the dynamics \eqref{eq:replicator_opt} are unique by the Picard-Lindel\"of theorem. The potential $\wh{J}$ is a strict Lyapunov function for replicator dynamics and unique forward trajectories converge to restricted equilibria \citep{Hofbauer:2000, Sandholm:2001}. Since replicator dynamics do not satisfy Nash stationarity, there may be restricted equilibria which are not Nash equilibria. However, no internal trajectory converges to any of these points \citep{Bomze:1986}. The solution trajectories of \eqref{eq:replicator_opt} are internal trajectories because $p_0$ is an internal point and $\mc{S}_c$ is invariant under all replicator dynamics with Lipschitz payoff function for finite time, as is clear from e.g. the tangent space parameterization \eqref{eq:tangent_param}.
\end{proof}

There is a simple relationship between potential functions in the multi-population and single-population settings.

\vspace{0.25cm}
\begin{lemma}[\textbf{Potential Embedding}]\label{lem:potential_embedding}
If $F\colon \WM\to T_0\WM$ has potential $J$, then $\wh F = Q \circ F\circ \Margin$ has potential $\wh{J} = J\circ \Margin$.
\end{lemma}
\begin{proof}
For $\wh{J}(p) = (J\circ \Margin)(p)$, we directly compute
\begin{equation}
    \nabla \wh{J}(p) = (D\Margin(p))^\top\nabla J(W) = (\Margin)^\top \circ \nabla J(W) = (Q \circ \nabla J \circ \Margin)(p)
\end{equation}
by denoting $W = \Margin(p)$ and using Lemma~\ref{lem:Q_adjoint}.
\end{proof}

We can now use the embedded potential of Lemma~\ref{lem:potential_embedding} and embedded Nash equilibria of Theorem~\ref{theorem:embedded_nash} to generalize the findings of Theorem~\ref{theorem:replicator_opt} to multiple populations. 

\vspace{0.25cm}
\begin{theorem}[\textbf{Multi-Population Replicators converge to NE}]\label{theorem:replicator_opt_multi}
Let $J\colon\WM\to\R$ be a $C^1$ potential such that the induced payoff function $F = \Pi_0\nabla J$ is Lipschitz on $\WM$. Then, for any internal point $W_0\in\WM$, the multi-population replicator dynamics
\begin{equation}\label{eq:replicator_opt_multi}
    \dot W(t) = \ROW_{W(t)}[F(W)],\qquad W(0) = W_0
\end{equation}
converge to a Nash equilibrium.
\end{theorem}
\begin{proof}
Let $p(t) = T(W(t))$. Then $p(t)$ follows the single-population replicator dynamics \eqref{eq:replicator_ms} by Theorem~\ref{theorem:af_embedding} which are induced by the embedded potential $\wh{J}(p) = J\circ \Margin$ due to Lemma~\ref{lem:potential_embedding} and start at the interior point $T(W_0)$ of $\Sn$. By Theorem~\ref{theorem:replicator_opt}, $p(t)$ converges to a NE of $\wh{F} = \Pi_0\circ Q\circ F\circ \Margin$ on $\Sn$. Since $p(t) = T(W(t))\in \imT$ for all times $t$, the limit point necessarily lies in the closure of $\imT$. Theorem~\ref{theorem:embedded_nash} then shows the assertion.
\end{proof}

By \cite[Proposition~3.1]{Sandholm:2001} all Nash equilibria satisfy the KKT optimality conditions for maximizing $J$ subject to simplex constraints.
If $J$ is concave, the KKT conditions are sufficient optimality conditions and thus \eqref{eq:replicator_opt} converges to a local maximizer.
In addition, $W$ is a Nash equilibrium exactly if $F(W)$ lies in the normal cone of the state space at $W$ \citep{Harker:1990, Nagurney:1998}. Thus, convergence of \eqref{eq:replicator_opt} to a boundary point which is not an extremal point only occurs if the trajectory reaches the boundary exactly perpendicularly. For assignment flows, it has been known that convergence to a non-extremal point of $\WM$ is an unusual occurrence. In fact, this behavior is not observed at all in the numerical solution of labeling problems for real-world data. \cite{Astroem2017} thus conjectured that convergence to a non-extremal point only occurs for a null set of initial population states. This was shown to be true for non-negative, linear fitness functions derived from a quadratic potential \citep{Zern:2020aa}.
From a game-theoretical perspective, only extremal points can be ESS under the posed conditions.

Note that the content of Theorem~\ref{theorem:replicator_opt_multi} is likely known to experts. We present it here to illustrate the power of the proposed formalism around Theorem~\ref{theorem:af_embedding} which provides a mathematical toolset for reducing the analysis of multi-population replicator dynamics to single-population ones.

\section{Conclusion}\label{sec:conclusion}
The proposed embedding framework for multi-population replicator dynamics provides a robust mathematical toolset for modeling complex population interactions. It formally reduces the complex multi-population case to a single-population one, simplifying subsequent analysis.
Current developments in the framework of assignment flows suggest multiple extensions of the present work.

In \cite{Savarino:2023}, assignment flows are characterized as critical points of an action functional within a geometric formalism of mechanics. An analogous characterization was previously suggested for single-population replicator dynamics \citep{raju2018variational} under assumptions which are valid only in the special case $n=2$ \cite[Section 4.4]{Savarino:2023}. In light of the present paper, a natural question is whether both perspectives are equivalent under embedding of the multi-population case.

A generalized perspective on assignment flows was proposed by \cite{Schwarz:2023}. The authors study a dynamical system on a product of density matrix manifolds called \emph{Quantum State Assignment Flow}. Although density matrices can represent entangled states and constitute a strict generalization of discrete probability measures, the underlying information geometric framework is broadly analogous. This suggests that generalized, quantum embedding results, along the lines introduced in the present paper, should be achievable. 

We briefly elaborate this point. Quantum mechanics is formulated on complex projective space $\mathbb{P}(\mathbb{C}^{c}) = \mathbb{P} ^{c-1}$, i.e. the state of an $c$-dimensional quantum system lives in the $(c-1)$-dimensional projective space. The states of a composite quantum system with $n$ components comprise the space $\mathbb{P}\big( \bigotimes_{i=1}^n \mathbb{C}^c \big) \cong \mathbb{P}^{N - 1}$, the projective space of the tensor products and $N = c^n$. The Segre embedding $\sigma$ \citep{Smith:2000aa} is an analytic isometric embedding of products of projective spaces into higher dimensional projective space \citep{Chen:2013}, i.e.
\begin{equation}
    \sigma : \mathbb{P}^{c-1} \times \dots \times \mathbb{P}^{c-1} \hookrightarrow \mathbb{P}^{N-1}, 
\end{equation}
where the product contains $n$ copies of $\mathbb{P} ^{c-1}$. The map $\sigma$ is an isometry if the product of projective spaces is equipped with the product of Fubini-Study metrics and the projective space of the tensor product with the high dimensional Fubini-Study metric. The \emph{separable (unentangled) quantum states} of the composite system are precisely the image of the Segre embedding. 
Furthermore, $\mathbb{P}^{c-1}$ admits a description as a toric variety with base $\triangle_c$, the simplex with boundary \citep{Bengtsson:2017aa}. Exploiting this structure makes it possible to choose compatible smooth embeddings
\begin{equation}
    \iota : \mathcal{S}_c \times  \dots  \times \mathcal{S}_c \hookrightarrow \mathbb{P}^{c-1} \times \dots  \times \mathbb{P}^{c-1},
\end{equation}
and to define a projection map $\pi_c:\mathbb{P}^{c-1} \to \triangle_c$. The embedding $T$ given by \eqref{eq:T_def} is compatible with the Segre embedding $\sigma$ in the sense that the diagram
\begin{equation}
    \begin{tikzcd}
        \mathbb{P}^{c-1} \times  \dots  \times  \mathbb{P}^{c-1} \ar[r,"\sigma"] 
        & 
        \mathbb{P}^{N-1} \ar[d,"\pi_N"] \\ 
        \mathcal{S}_c \times \dots \times \mathcal{S} _c \ar[u,hook,"\iota"] \ar[r,"T"] 
        & 
        \mathcal{S}_N  
    \end{tikzcd}
\end{equation}
is well-defined and commutes. Well-defined refers here to the fact that $\pi_N \circ \sigma \circ \iota$ maps to $\mathcal{S} _N \subset \triangle_N$. Additionally, similar compatibility relations remain valid when quantum states are described in terms of density matrices. 

Elaborating the consequences of the results in this paper in connection with the more general quantum state assignment flow approach is an attractive research problem for future work.

\backmatter

\vspace{1cm}
\bmhead{Acknowledgments}

The original artwork used in Figure~\ref{fig:cartoon_labeling} was designed by dgim-studio / Freepik.

The authors thank Fabrizio Savarino for many fruitful discussions and suggestions.

\section*{Declarations}

\paragraph{Funding}
This work is funded by the Deutsche Forschungsgemeinschaft (DFG), grant SCHN 457/17-1, within the priority programme SPP 2298: ``Theoretical Foundations of Deep Learning''.
This work is funded by the Deutsche Forschungsgemeinschaft (DFG) under Germany's Excellence Strategy EXC-2181/1 - 390900948 (the Heidelberg STRUCTURES Excellence Cluster).

\paragraph{Competing interests}
The authors have no relevant financial or non-financial interests to disclose.

\newpage
\bibliography{journal_game}

\newpage
\begin{appendices}

\section{Additional Lemmata}\label{app:lemmata}

\begin{lemma}\label{lem:T_injective}
The mapping $T\colon \WM\to \imT$ defined by \eqref{eq:T_def} is injective.
\end{lemma}
\begin{proof}
Let $W^{(1)}, W^{(2)}\in \WM$ satisfy $T(W^{(1)}) = T(W^{(2)})$. Let $\gamma \in [c]^n$ be an arbitrary multi-index. Fix an arbitrary vertex $i\in [n]$ and let $\alpha\in [c]^n$ match $\gamma$ at all vertices $k\in [n]\setminus\{i\}$. Then $T(W^{(1)}) = T(W^{(2)})$ implies both $T(W^{(1)})_\gamma = T(W^{(2)})_\gamma$ and $T(W^{(1)})_\alpha = T(W^{(2)})_\alpha$. Division thus gives 
\begin{equation}\label{eq:t_inj_proof1}
	W^{(1)}_{i,\alpha_i} W^{(2)}_{i,\gamma_i} = W^{(1)}_{i,\gamma_i} W^{(2)}_{i,\alpha_i}.
\end{equation}
Since $W^{(1)}, W^{(2)}\in \WM$, the entries of row $i$ sum to 1. Using this and the fact that $\alpha_i\in [c]$ is arbitrary, we find
\begin{equation}
	W^{(2)}_{i,\gamma_i} 
	= \sum_{j\in [c]} W^{(1)}_{i,j} W^{(2)}_{i,\gamma_i}
	\stackrel{\eqref{eq:t_inj_proof1}}{=} 
	\sum_{j\in [c]} W^{(1)}_{i,\gamma_i} W^{(2)}_{i,j}
	= W^{(1)}_{i,\gamma_i}.
\end{equation}
Since $\gamma_i\in [c]$ was arbitrary, this shows $W^{(1)} = W^{(2)}$.
\end{proof}

\begin{lemma}\label{lem:embedded_support}
For every $W\in\ol{\WM}$ one has $\gamma \in \supp (T(W))$ if and only if $\gamma_i\in\supp (W_i)$ for all $i\in[n]$.
\end{lemma}
\begin{proof}
We directly compute
\begin{align}
    \supp(T(W)) &= \{\gamma\in [c]^n\colon T(W)_\gamma > 0\}\\
        &= \{\gamma\in [c]^n\colon \prod_{i\in[n]} W_{i,\gamma_i} > 0\}\\
        &= \{\gamma\in [c]^n\colon W_{i,\gamma_i} > 0,\;\forall i\in[n]\}\\
        &= \{\gamma\in [c]^n\colon \gamma_i \in \supp(W_i),\;\forall i\in[n]\}.
\end{align}
\end{proof}

\begin{lemma}\label{lem:Q_proj_commute}
For any $V\in\R^{n\times c}$ it holds $Q\Pi_0 V = \Pi_0 QV$ for the mappings $Q, \Pi_{0}$ given by \eqref{eq:def_Q} and \eqref{eq:projection_Pi0}.
\end{lemma}
\begin{proof}
For arbitrary $\gamma\in [c]^n$, we compute
\begin{equation}
	(Q\Pi_0 V)_\gamma 
		= \sum_{i\in[n]} (\Pi_{0} V_{i})_{\gamma_{i}}
		= \sum_{i\in [n]} \big( V_{i,\gamma_i} - \la V_i, \frac{1}{c}\eins_c\ra \big)
		= (QV)_\gamma - \la V, \underbrace{\frac{1}{c}\eins_{n\times c}}_{= \Margin \BS}\ra
\end{equation}
and thus, by Lemma~\ref{lem:Q_adjoint} ,
\begin{equation}
	Q\Pi_0 V = QV - \la QV, \frac{1}{N} \eins_N\ra\eins_N = \Pi_0 QV,
\end{equation}
which was the assertion.
\end{proof}

\begin{lemma}\label{lem:dT_closed_form}
  Let $\gamma \in [c]^n$. The differential of $T$ at $W \in \WM$ in direction $V \in T_0\WM$ is given by
  \begin{equation}\label{eq:dT_closed_form}
    dT|_W[V] = \big( dT_\gamma|_W [V]\big)_{\gamma \in [c]^n} = T(W) \diamond Q\Big[\frac{V}{W}\Big].
  \end{equation}
\end{lemma}
\begin{proof}
  Suppose $\eta \colon (-\veps, \veps) \to \WM$ is a smooth curve with $\eta(0) = W$ and $\dot{\eta}(0) = V$, for some $\veps>0$. Let $\gamma \in [c]^n$ be arbitrary and consider the component $T_\gamma$. Then
  \begin{subequations}
  \begin{align}
    d T_\gamma|_W[V] &= \tfrac{d}{dt} T_\gamma(\eta(t))\big|_{t=0} = \tfrac{d}{dt} \prod_{i \in [n]} \eta_{i, \gamma_i}(t)\big|_{t=0}\\
    &= \sum_{k\in[n]} \dot{\eta}_{k, \gamma_k}(0) \prod_{i\in[n]\setminus\{k\}} \eta_{i, \gamma_i}(0) = \sum_{k\in[n]} V_{k, \gamma_k} \prod_{i\in[n]\setminus\{k\}} W_{i, \gamma_i}\\
    &= \sum_{k\in[n]} \frac{V_{k, \gamma_k}}{W_{k, \gamma_k}} T_\gamma(W) \overset{\eqref{eq:def_Q}}{=} T_\gamma(W) Q_\gamma\Big(\frac{V}{W}\Big).
  \end{align}
  \end{subequations}
  Because of $dT|_W[V] = ( dT_\gamma|_W [V])_{\gamma \in [c]^n}$ the expression in \eqref{eq:dT_closed_form} directly follows.
\end{proof}

\begin{lemma}\label{lem:ker_Q}
	It holds $\mkernel Q = \{\Diag(d)\eins_{n\times c}\colon d\in\R^n,\;\la d,\eins_n\ra = 0\}$ as well as $\mrank Q = nc - (n-1)$.
\end{lemma}
\begin{proof}
	Let $V\in \mkernel Q$ and let $\gamma, \tilde \gamma$ be two multi-indices which differ exactly at position $k$ but are otherwise arbitrary. We have $(QV)_{\gamma} = (QV)_{\tilde\gamma} = 0$ because $V\in \mkernel Q$. Thus
	\begin{equation}
		(QV)_{\tilde\gamma}
		= V_{k,\tilde\gamma_k} + \sum_{i\in [n]\setminus \{k\}} V_{i, \tilde\gamma_i}
		= (QV)_{\gamma}
		= V_{k,\gamma_k} + \sum_{i\in [n]\setminus \{k\}} V_{i, \gamma_i}
	\end{equation}
	which implies $V_{k,\tilde\gamma_k} = V_{k,\gamma_k}$, i.e. $V = \Diag(d)\eins_{n\times c}$ for some $d\in\R^n$ since $k$ was arbitrary. Further, it holds
	\begin{equation}
		0 = (QV)_\gamma = \sum_{i\in [n]} V_{i,\gamma_i} = \sum_{i\in [n]} d_i = \la d, \eins_n\ra.
	\end{equation}
	Thus, we have shown
	\begin{equation}\label{eq:mkernel_Q_inclusion}
		\mkernel Q \subseteq \{\Diag(d)\eins_{n\times c}\colon d\in\R^n,\;\la d,\eins_n\ra = 0\}.
	\end{equation}
	Conversely, let $V$ be in the right-hand set. Then
	\begin{equation}
	(QV)_\gamma = \sum_{i\in [n]} V_{i,\gamma_i} = \sum_{i\in [n]} d_i = \la d, \eins_n\ra = 0
	\end{equation}
	for all $\gamma\in [c]^n$ which shows that \eqref{eq:mkernel_Q_inclusion} is an equation. % $V\in \mkernel Q$.
	There are $(n-1)$ linearly independent vectors $d\in \R^n$ with $\la d, \eins_n\ra = 0$, therefore $Q$ has the specified rank.
\end{proof}

\section{Proofs}\label{app:proofs}
\subsection{Proof of Theorem~\ref{prop:isometric_embedding_T}}\label{proof:prop:isometric_embedding_T}
\begin{theorem}[Theorem~\ref{prop:isometric_embedding_T} in the main text]\label{app:prop:isometric_embedding_T}
The map $T\colon \WM\to\imT$ is an isometric embedding of $\WM$ equipped with product Fisher-Rao geometry into $\Sn$ equipped with the Fisher-Rao geometry. On its image $T(\WM) =: \imT\subseteq \Sn$, the inverse is given by marginalization 
\begin{equation}\label{app:eq:T_inv}
  \Margin|_\imT = T^{-1} \colon \imT \to \WM.
\end{equation}
\end{theorem}
\begin{proof}
A standard argument (Lemma~\ref{lem:T_injective}) shows that $T\colon \WM\to \imT$ is injective.
We check that the inverse of $T$ has the shape \eqref{eq:T_inv}.
\begin{align}
  (\Margin T(W))_{i,j} 
    &= \sum_{\gamma\;:\; \gamma_i = j}\prod_{r\in [n]} W_{r,\gamma_r} 
    = \sum_{\gamma\;:\; \gamma_i = j}W_{i, j}\prod_{r\in [n]\setminus \{i\}} W_{r,\gamma_r} \\
    &= \sum_{l\in[n]\setminus\{i\}}\sum_{\gamma_{l}\in [c]}\prod_{r\in [n]\setminus \{i\}} W_{r,\gamma_r}\\
    &= W_{i,j} \sum_{k_1\in [c]} W_{1,k_{1}}\sum_{k_2\in [c]} W_{2,k_{2}}\dotsc \sum_{k_n\in [c]} W_{n,k_{n}}\\
    &= W_{i,j} \prod_{r\in [n]\setminus \{i\}} \underbrace{\sum_{\gamma_r \in [c]} W_{r, \gamma_r}}_{= 1} = W_{i,j}.
\end{align}
Clearly, all component functions of $T$ and $T^{-1}$ are smooth.
We will now show that $T$ is a topological embedding, i.e. a homeomorphism with respect to the subspace topology of $\imT\subseteq \Sn$.
Let 
\begin{equation}\label{eq:img_Q_tangent}
\mc{Q} = Q(T_0\WM)
\end{equation}
denote the image of $T_0\WM$ under $Q$. $\mc{Q}$ is a linear subspace of $T_0\Sn$ because, for any $V\in T_0\WM$, we have
\begin{equation}
  QV = Q\Pi_0 V = \Pi_0 QV \in T_0\Sn
\end{equation}
by Lemma~\ref{lem:Q_proj_commute}. In addition, Lemma~\ref{lem:ker_Q} shows $\mkernel Q \cap T_0\WM = \{0\}$, since any matrix in $\ker Q$ has constant row vectors. Thus, the restriction of $Q$ to $T_0\WM$ is injective and since $T_0\WM$ and $\mc{Q}$ have finite dimension, $Q|_{T_0\WM}$ is a homeomorphism.
The lifting map at the barycenter is the inverse of the \emph{global} $e$-coordinate chart of information geometry up to a change of basis. In particular, $\exp_{\BW}\colon T_0\WM\to \WM$ and $\exp_{\eins_{\mc{S}_{N}}}\colon T_0\Sn\to \Sn$ are homeomorphisms.
Now let
\begin{equation}
  \psi\colon \imT\to \mc{Q},\qquad p\mapsto \psi(p) = \exp_{\eins_{\mc{S}_{N}}}^{-1}(p)
\end{equation}
which is well-defined due to Lemma~\ref{lem:lifting_map} and denote the initial topology of $\imT$ with respect to $\psi^{-1}$ by $\mc{A}$. Then $T$ is a homeomorphism of $\WM$ and $\imT$ equipped with the topology $\mc{A}$ because
\begin{equation}
  T = \exp_{\eins_{\mc{S}_{N}}}\circ Q|_{T_0\WM} \circ \psi^{-1}
\end{equation}
by Lemma~\ref{lem:lifting_map}.
It remains to show that $\mc{A}$ coincides with the subspace topology of $\imT\subseteq\Sn$. Note that the topology of $\mc{Q}$ is the subspace topology of $\mc{Q}\subseteq T_0\Sn$ and recall that $\exp_{\eins_{\mc{S}_{N}}}\colon T_0\Sn\to \Sn$ is a homeomorphism. For a subset $A\subseteq \mc{Q}$ we thus have
\begin{subequations}
\begin{align}
  A\in \mc{A}\,&\Leftrightarrow\, \psi(A)\text{ is open in }\mc{Q}\\
  &\Leftrightarrow\, \exp_{\eins_{\mc{S}_{N}}}^{-1}(A) = B \cap \mc{Q}\text{ for an open set }B\subseteq T_0\Sn\\
  &\Leftrightarrow\, \exp_{\eins_{\mc{S}_{N}}}^{-1}(A) = \exp_{\eins_{\mc{S}_{N}}}^{-1}(\ol A) \cap \mc{Q}\text{ for an open set }\ol A\subseteq \Sn\\
  &\Leftrightarrow\, A = \ol A \cap \exp_{\eins_{\mc{S}_{N}}}(\mc{Q})\text{ for an open set }\ol A\subseteq \Sn\\
  &\Leftrightarrow\, A = \ol A \cap \imT \text{ for an open set }\ol A\subseteq \Sn.
\end{align}
\end{subequations}
This shows that $\mc{A}$ is the subspace topology of $\imT\subseteq \Sn$ and thus, $T$ is a topological embedding of $\WM$ into $\Sn$.

We compute the rank of $T$ by applying Lemma~\ref{lem:dT_closed_form}. Let $W\in\mc{W}$ and $V\in T_0\mc{W}$ be in the kernel of $dT|_W$. Then
\begin{equation}
  0 = dT|_W[V] = T(W)\diamond Q\Big[\frac{V}{W}\Big]
\end{equation}
which implies $\frac{V}{W} \in \mkernel Q$ because $T(W)_\gamma \neq 0$ for all $\gamma\in [c]^n$. By Lemma~\ref{lem:ker_Q} this implies
\begin{equation}\label{eq:embedding_rank_v}
  V = W\diamond (\Diag(d)\eins_{n\times c}) = \Diag(d)W
\end{equation}
for some $d\in\R^n$ with $\la d,\eins_n\ra = 0$. From $V\in T_0\mc{W}$ we find
\begin{equation}
  0 = \la V_i, \eins_c\ra = d_i\la W_i, \eins_c\ra = d_i,\qquad \forall i\in [n]
\end{equation}
which shows $V = 0$ by \eqref{eq:embedding_rank_v}, i.e. $dT|_W$ has full rank. Thus, $T$ is an injective immersion.

It remains to show that $T$ is metric compatible.
Suppose $W \in \WM$ and  $U, V \in T_0\WM$ are arbitrary. Denoting the Fisher-Rao metric on $\Sn$ by $g^{\SM_N}$ we get
\begin{subequations}
\begin{align}
  (T^* g^{\SM_N}\big)_W(U, V) &\overset{\phantom{\eqref{eq:dT_closed_form}}}{=} g^{\SM_N}_{T(W)}(dT|_W[U], dT|_W[V])\\
  &\overset{\phantom{\eqref{eq:dT_closed_form}}}{=} \big\la dT|_W[U], \tfrac{\eins}{T(W)}\diamond dT|_W[V] \big\ra\\
  &\overset{\eqref{eq:dT_closed_form}}{=} \Big\la dT|_W[U], Q\Big[\frac{V}{W}\Big] \Big\ra\\
  &\overset{\phantom{\eqref{eq:dT_closed_form}}}{=}\Big\la \Margin dT|_W[U], \frac{V}{W}\Big\ra.\label{eq:prop:T_is_metric_compatible:eq1}
\end{align}
\end{subequations}
Note that $\Margin$ is linear, implying $d\Margin|_p = \Margin$ for every $p \in \SM_N$. Since $\Margin$ restricted to $\imT = T(\WM)$ is the inverse of $T$, one ahs $\Margin \circ T = \mrm{id}_{\WM}$. These two facts imply
\begin{equation}
  \Margin\big[ dT|_W[U]\big] = d\Margin|_{T(W)}\big[ dT|_W[U]\big] = d\big(\Margin\circ T\big)|_W[U] = d (\mrm{id}_{\WM})|_W [U] = U.
\end{equation}
Plugging this result back into \eqref{eq:prop:T_is_metric_compatible:eq1} gives
\begin{equation}
  (T^* g^{\SM_N}\big)_W(U, V) = \Big\la U, \frac{V}{W}\Big\ra = g^\WM_W(U, V)
\end{equation}
which shows the assertion.
\end{proof}

\subsection{Proof of Proposition~\ref{prop:max_entropy}}\label{proof:prop:max_entropy}
\begin{proposition}[Proposition~\ref{prop:max_entropy} in the main text]\label{app:prop:max_entropy}
For every $W\in\WM$, the distribution $T(W)\in\Sn$ has maximum entropy among all $p\in \Sn$ subject to the marginal constraint $\Margin p = W$.
\end{proposition}
\begin{proof}
We use the concepts of m-flat and e-flat submanifolds of information geometry, which justify applying the Pythagorean relation of information geometry.
For details, we refer to \cite{amari2007methods}.
The feasible set of all distributions with the prescribed marginals reads
\begin{equation}\label{eq:max_entropy_feasible}
\{ T(W) + u\colon Mu = 0\} \cap \Sn
\end{equation}
which is an m-flat submanifold of $\Sn$. In addition, Lemma~\ref{lem:lifting_map} shows that $\imT$ is an e-flat submanifold of $\Sn$.
Let $p = T(W) + u$ denote an arbitrary feasible point.
By \eqref{eq:max_entropy_feasible} and Lemma~\ref{lem:Q_adjoint} we have
\begin{equation}\label{eq:max_entropy_orthog1}
  \la u, QV\ra = \la \Margin u, V\ra = 0
\end{equation}
for all $V \in \R^{n\times c}$. 
Consider the $m$-geodesic connecting $p$ with $T(W)$. It intersects $\imT$ at $T(W)$ and we find
\begin{equation}\label{eq:m_geod_angle}
  \la dT|_W[V], u\ra = \la T(W) \diamond Q\Big[\frac{V}{W}\Big], u\ra_{T(W)}
  = \la Q\Big[\frac{V}{W}\Big], u\ra
  = \la \frac{V}{W}, \Margin u\ra
  = 0
\end{equation}
by using Lemma~\ref{lem:dT_closed_form}. With \eqref{eq:m_geod_angle}, m-flatness of \eqref{eq:max_entropy_feasible} and e-flatness of $\imT$ the prerequisites for the Pythagorean relation of information geometry \cite[Theorem~3.8]{amari2007methods} are met.
Using the cross-entropy $H(p, q) = -\la p, \log q\ra$ as well as the relative entropy $\KL(p, q) = \la p, \log \frac{p}{q}\ra$ and barycenter $\eins_{\mc{S}_{N}} = \frac{1}{N}\eins$, we find
\begin{align}\label{eq:entropy_tW}
  H(T(W)) &= H(T(W), \eins_{\Sn}) - \KL(T(W), \eins_{\Sn}) \nonumber\\
  &= \log N - \KL(T(W), \eins_{\Sn})
\end{align}
and consequently
\begin{align}
  H(p) &= H(p, \eins_{\Sn}) - \KL(p, \eins_{\Sn}) \\
    &= \log N - \KL(p, \eins_{\Sn})\\
    &= H(T(W)) + \KL(T(W), \eins_{\Sn}) - \KL(p, \eins_{\Sn})\\
    &\stackrel{(\ast)}{=} H(T(W)) + \KL(T(W), \eins_{\Sn}) - \KL(p, T(W)) - \KL(T(W), \eins_{\Sn})\\
    &= H(T(W)) - \KL(p, T(W))
\end{align}
by the Pythagorean relation $(\ast)$. 
Therefore $H(p) \leq H(T(W))$ with equality only for $p = T(W)$ which shows the assertion.
\end{proof}

\subsection{Proof of Lemma~\ref{lem:lifting_map}}\label{proof:lem:lifting_map}
\begin{lemma}[\textbf{Lifting Map Lemma}]\label{app:lem:lifting_map}
  Let $S\in \WM$ and $V\in \R^{n\times c}$. Then
  \begin{equation}\label{app:eq:lifting_map_delta}
    T(\exp_{S}(V)) = \exp_{T(S)}(Q(V)).
  \end{equation}
\end{lemma}
\begin{proof}
  We have $T(\exp(V)) = \exp(Q(V))$ (without subscripts, i.e. applying the exponential function componentwise), because for any multi-index $\gamma$
  \begin{subequations}\label{eq:lem_lifting_Qexp}
  \begin{align}
    \exp(Q(V))_\gamma 
    &= \exp\big(Q(V)_{\gamma}\big)
    = \exp\Big( \sum_{i\in [n]} V_{i,\gamma_i} \Big)\\
    &= \prod_{i\in [n]}\exp\Big( V_{i,\gamma_i} \Big) 
    = \prod_{i\in[n]}\big(\exp(V)\big)_{i,\gamma_{i}}\\
    &= T(\exp(V))_\gamma.
  \end{align}
  \end{subequations}
  Let $D\in \R^{n\times n}$ be a diagonal matrix with nonzero diagonal entries. Then $T(DR) \propto T(R)$ for any $R\in\R^{n\times c}$ because
  \begin{equation}\label{eq:lem_lifting_pixelscaling}
    T(DR)_\gamma = \prod_{i\in [n]} (DR)_{i,\gamma_i} = \Big(\prod_{i\in [n]} D_{ii}\Big)\Big(\prod_{i\in [n]}R_{i,\gamma_i}\Big) \propto T(R)_\gamma.
  \end{equation}
  It follows that
  \begin{equation}\label{eq:lem_lifting_proarg}
    T(\exp_{S}(V)) \propto T(S\diamond\exp(V)) \stackrel{\eqref{eq:lem_lifting_Qexp}}{=} T(S)\diamond \exp(Q(V)) \propto \exp_{T(S)}(Q(V))
  \end{equation}
  Because both the first and last term in \eqref{eq:lem_lifting_proarg} are clearly elements of $\Sn$, i.e. strictly positive vectors summing up ot 1, this implies the assertion.
\end{proof}

\subsection{Proof of Theorem~\ref{theorem:af_embedding}}\label{proof:theorem:af_embedding}

\begin{theorem}[\textbf{Multi-Population Embedding Theorem}]\label{app:theorem:af_embedding}
For any payoff function $F\colon \WM\to \R^{n\times c}$, the multi-population replicator dynamics
\begin{equation}\label{app:eq:general_af}
  \dot W = \ROW_W[F(W)], \quad W(0) = W_0
\end{equation}
on $\WM$ is pushed forward by $T$ to the replicator dynamics
\begin{equation}\label{app:eq:replicator_ms}
  \dot p(t) = \ROS_p(t)\wh F(p(t)), \quad p(0) = T(W_0), \quad \wh F = Q \circ F \circ \Margin
\end{equation}
on $\Sn$ and the map $T$ satisfies
\begin{equation}\label{eq:dT_ReplOp}
  dT|_W[\ROW_W[X]] = R_{T(W)} Q[X], \quad \text{for all } X \in \R^{n \times c} \text{ and } W\in \WM.
\end{equation}
\end{theorem}
\begin{proof}
We first show that, for any $W \in \WM$, the differential of $T$ and the replicator operator are related by \eqref{eq:dT_ReplOp}.
Let $\gamma \in [c]^n$ be an arbitrary multi-index. Because of $\mc{R}_W[X] \in \TW$, Lemma~\ref{lem:dT_closed_form} implies
\begin{equation}
  dT_\gamma|_W[\ROW_W[X]] = T_\gamma(W)Q_\gamma\bigg[ \frac{\ROW_W[X]}{W} \bigg] = T_\gamma(W) \sum_{i\in [n]}\frac{(\ROW_W[X])_{i, \gamma_i}}{W_{i, \gamma_i}}.
\end{equation}
Due to $(\ROW_W[X])_{i, \gamma_i} = W_{i, \gamma_i}(X_{i, \gamma_i} - \la X_i, W_i\ra)$, the sum can be written as
\begin{equation}
  \sum_{i\in [n]}\frac{(\ROW_W[X])_{i, \gamma_i}}{W_{i, \gamma_i}} = \sum_{i\in [n]}\big(X_{i, \gamma_i} - \la X_i, W_i\ra\big) = Q_\gamma[X] - \la X, W\ra.
\end{equation}
Additionally using the relation $W = \Margin[T(W)]$ due to \eqref{app:eq:T_inv}, and applying Lemma~\ref{lem:Q_adjoint} gives 
\begin{equation}
  \la X, W\ra = \la X, \Margin[T(W)]\ra = \la Q[X], T(W)\ra.
\end{equation}
Collecting all expressions, we have
\begin{equation}
  dT_\gamma|_W[\ROW_W[X]] = T_\gamma(W)\big(Q_\gamma[X] - \la Q[X], T(W)\ra\big) = \big(R_{T(W)}Q[X]\big)_\gamma
\end{equation}
which shows \eqref{eq:dT_ReplOp}.
Now, denoting $\p = T(W)\in \Sn$ we directly establish \eqref{app:eq:replicator_ms}
\begin{equation}
  \dot \p = dT(W)[\ROW_W[(F\circ \Margin)(\p)]] = \ROS_\p[(Q\circ F\circ \Margin)(\p)] = \ROS_\p[\wh F(\p)].
\end{equation}
\end{proof}

\end{appendices}

\end{document}